\documentclass{amsart}

\usepackage{amsmath, amsthm, amssymb}  % 수학 기호 및 정리 환경
\usepackage{mathtools}                % amsmath 확장, 특히 \coloneqq 등
\usepackage{enumitem}                 % 리스트 환경 커스터마이징
\usepackage{hyperref}                 % 링크 및 참조 기능
\usepackage{mathrsfs}                 % \mathscr 폰트
\usepackage{tikz}                     % 도식 그릴 때 (선택)
\usepackage{url}                      % URL 표현할 때
\usepackage{cleveref}
\usepackage{thm-restate}
\usepackage{setspace}

\newtheorem{theorem}{Theorem}[section]
\newtheorem{lemma}[theorem]{Lemma}
\newtheorem{proposition}[theorem]{Proposition}

\theoremstyle{definition}
\newtheorem{definition}[theorem]{Definition}

\newtheorem{question}[theorem]{Question}

\theoremstyle{remark}
\newtheorem{remark}[theorem]{Remark}

\numberwithin{equation}{section}

\newcommand{\Z}{\mathbb{Z}}
\newcommand{\N}{\mathbb{N}}
\newcommand{\qa}{q_{\mathrm{accept}}}
\newcommand{\qr}{q_{\mathrm{reject}}}

\AtBeginDocument{
    \setlength{\abovedisplayskip}{13pt}
    \setlength{\belowdisplayskip}{13pt}
    \setlength{\abovedisplayshortskip}{13pt}
    \setlength{\belowdisplayshortskip}{13pt}
}

\begin{document}

\title[Undecidability of Finite Orbit Recognition in Polynomial Maps]{Undecidability of Finite Orbit Recognition in Polynomial Maps}

\author{Gwangyong Gwon}
\address{Department of Mathematical Science, Seoul National University, Seoul, Korea}
\email{bluelily@snu.ac.kr}

\begin{abstract}
We prove that orbit finiteness is undecidable even for a restricted class of finitely presented local polynomial maps on $\Z^\N$.
The proof gives a direct encoding of deterministic Turing machine dynamics into polynomial dynamics.%논문의 핵심 요약. AMS 스타일은 abstract가 title 이후, 본문 이전에 위치.
\end{abstract}

\maketitle
\section{Introduction}%서론. 문제 제기, 배경, 관련 연구.
\subsection{Main results}
This paper concerns several undecidability results in computability and dynamics.
Suppose that \(S\) is a finite collection of finitely presented polynomial maps on \(\Z^\N\).
We shall say that a point \(x=(x_i)\in\Z^\N\) is \(S\textrm{-}stable\) if its \(S\textrm{-}\)orbit
\[\mathcal{O}_S(x)=\{F_k\circ\cdots\circ F_1(x):k\ge0,\ F_i\in S\}\]
is a finite set.
For example, if \(S\) consists of a single map $F$, then a point $x\in\Z^\N$ is $S$-stable if and only if $F^k(x)=F^l(x)$ for some $k>l\ge0$.
Also, we  say that a trajectory of a Turing machine on a given input is \textit{eventually periodic} if the machine eventually returns to a previously visited configuration.
To induce our main result, we begin with an analogy of the Halting Problem \cite{turing1936}.

\begin{restatable}{theorem}{TEP}\label{und_tep}
There is no algorithm which, given a deterministic one-tape Turing machine $M$ and a finite input word $w$, decides whether the trajectory of $M$ on $w$ is eventually periodic.
\end{restatable}

\Cref{und_tep} is, in fact, an instance of a general undecidability result;
any nontrivial limit property of Turing machines cannot be determined by an algorithm, such as ``Does a Turing machine yield a trajectory which contains a given configuration?" or ``Does a Turing machine write a certain symbol infinitely often?"\\
Meanwhile, this implies, through reduction, our main theorem.

\begin{restatable}[Orbit-finiteness theorem]{theorem}{FO}\label{und_fo}
There is no algorithm which, given a finite set \(S\) of finitely presented local polynomial maps on \(\Z^\N\) and a finitely supported point \(x\in\Z^\N\), decides whether \(x\) is \(S\)-stable.
This remains true even when \(S=\{F\}\) is a singleton.
\end{restatable}

Our proof combines tools from computability and algebra, illustrating that orbit structure questions in polynomial maps are algorithmically intractable in general.

\subsection{Historical Backgrounds}
For a set $X$ and a collection $S$ of endomorphisms of $X$, we say a point $x\in X$ is $S$-\textit{periodic} if its $S$-orbit $\mathcal{O}_S(x)=\{f(x):f\in\langle S\rangle\}$ is a finite set and the endomorphisms in $S$ act on $\mathcal{O}_S(x)$ by permutations.
Whang \cite{whang2023} proved a decidability result about periodic orbits.
\begin{theorem}[Whang, \cite{whang2023}]
Fix $n\ge1$.
There is an effective universal constant $C(n)$ such that, for an arbitrary set $S$ of polynomial maps from $\Z^n$ to itself with integer coefficients, and any $S$-periodic point $x\in\Z^n$, we have
\begin{equation*}
|\mathcal{O}_S(x)|\le C(n).
\end{equation*}
In particular, there is a polynomial-time algorithm to decide, for $x\in\Z^n$ and a finite set $S$ of polynomial endomorphisms of $\Z^n$ over $\Z$, whether or not $x$ is $S$-periodic.
\end{theorem}
More generally, he established the following version.
\begin{theorem}[Whang, \cite{whang2023}]
Let $S$ be a finite set of endomorphisms of an algebraic variety $V/\bar{\mathbb{Q}}$.
There is an algorithm to decide, given $x\in V(\bar{\mathbb{Q}})$, whether or not $x$ is $S$-periodic.
\end{theorem}
Meanwhile, Whang posed a question about finite orbits.
\begin{question}[Whang, \cite{whang2023}]
Given a finite set $S$ of endomorphisms of an algebraic variety $V/\bar{\mathbb{Q}}$, is the subset of points with finite $S$-orbit decidable?
\end{question}
We shall approach this question within a different space $\Z^\N$, and conclude that the answer is negative.
In contrast, Kim \cite{kim2025} proved that for an algebraic variety $V/\bar{\mathbb{Q}}$ there is an algorithm which decides a given point $x\in V(\bar{\mathbb{Q}})$ and a finite collection of unramified endomorphisms yield a finite orbit.

\subsection{Acknowledgements}
I thank Junho Peter Whang for beneficial advice and comments.

\section{Finite input format}

\subsection{Uniform coordinate coding}
We fix here the coordinate convention once and for all.
\begin{definition}[Uniform block coding]
For each integer $r\ge1$, fix a computable bijection
\[
\iota_r:\Z\times\{1,\cdots,r\}\to\N
\]
such that both maps
\[
(r,i,j)\mapsto\iota_r(i,j),\quad(r,n)\mapsto\iota_r^{-1}(n)
\]
are computable.
We interpret $i\in\Z$ as a tape-cell coordinate and $j\in\{1,\cdots,r\}$ as a coordinate inside the cell block.
\end{definition}
This removes the ambiguity of unused coordinates.
Every coordinate of $\Z^\N$ is covered exactly once by the coding $\iota_r$.

\subsection{Finitely presented local polynomial maps}
\begin{definition}[Finitely presented local polynomial map]
A finitely presented local polynomial map on $\Z^\N$ is specified by finite data
\[(r,R,P_1,\cdots,P_r),\]
where $r\ge1,\ R\ge0,$ and
\[P_j\in\Z[Y_{u,h}:-R\le u\le R,1\le h\le r]\qquad(1\le j\le r).\]
It defines a map $F:\Z^\N\to\Z^\N$ by
\[
F_{\iota_r(i,j)}(x)=P_j(x_{\iota_r(i+u,h)}:-R\le u\le R,1\le h\le r)
\]
for every $i\in\Z$ and $1\le j\le r$.
\end{definition}
Each polynomial $P_j$ is encoded as a finite list of monomials.
A monomial is given by its integer coefficient and its exponent vector.
Hence the tuple $(r,R,P_1,\cdots,P_r)$ is a finite string.

\begin{remark}
The general input class in \Cref{und_fo} does not require maps in $S$ to preserve finite support.
For example, a nonzero constant term may produce an infinite-support point from the zero point.
This is not a problem: orbit finiteness is still a well-defined equality question in $\Z^\N$.
The reduction below uses a special zero-preserving subclass for which finite support is preserved.
\end{remark}

\section{Deterministic Turing machines}

\subsection{Turing machines}\label{sec:tm}
We follow \cite{Sipser2012} for basic definitions.
Informally, a \emph{Turing machine} consists of a two-sided tape of infinite length, a read--write head, and a finite table of instructions that determine how the head moves, reads, and writes symbols on the tape.
The machine operates in discrete time steps, at each step updating its state, writing a symbol, and moving the head one cell left or right.

\begin{definition}
	A deterministic Turing machine is a 7-tuple
	\[M=(Q, \Sigma, \Gamma, \delta, q_0, \qa, \qr),\]
	where $Q, \Sigma, \Gamma$ are all finite sets and
	\begin{itemize}
		\item $Q$ is the set of states,
		\item $\Sigma$ is the input alphabet not containing the blank symbol $\sqcup$,
		\item $\Gamma$ is the tape alphabet, where $\sqcup \in \Gamma$ and $\Sigma\subseteq\Gamma$,
		\item $\delta: (Q\setminus\{\qa,\qr\})\times\Gamma \to Q \times \Gamma \times \{\mathrm{L},\mathrm{R}\}$ is the transition function,
		\item $q_0 \in Q$ is the start state,
		\item $\qa, \qr \in Q$ are the halting states, with $\qa \neq \qr$.
	\end{itemize}
\end{definition}
The machine receives its input word $w$ over $\Sigma$, written from the \(0\)th tape cell to the right, with all other cells containing the blank symbol $\sqcup$.
The head begins on the leftmost square of the tape at the state $q_0$.
Computation proceeds according to the transition function $\delta$.
The machine halts when it enters either $\qa$ or $\qr$; otherwise it continues indefinitely.

\begin{definition}\label{def_ct}
A configuration consists of a tape function $\tau:\Z\to\Gamma$ with finite nonblank support, a head position $h\in\Z$, and a state $q\in Q$.
The associated trajectory map $T_M$ follows $\delta$ on nonhalting configurations and fixes halting configurations:
\[q\in\{\qa,\qr\}\to T_M(C)=C.\]
Thus halting trajectories are eventually fixed.
\end{definition}
\begin{definition}
The trajectory of $M$ from an initial configuration $C_0$ is eventually periodic if there exist integers $k<l$ such that
\[T^k_M(C_0)=T^l_M(C_0).\]
Since $T_M$ is deterministic, this is equivalent to periodicity of the tail after time $k$.
\end{definition}
\Cref{und_tep} concerns the trajectory of configurations of a Turing machine under successive applications of the transition function.

\subsection{Undecidability of Turing trajectory eventual periodicity}\label{sec:uttep}
We now introduce the classical \emph{Halting Problem}, following \cite{Sipser2012} and \cite{hamkins2024turing}.
Let
\[
\mathrm{HALT}=\{ \langle M, w \rangle : M \text{ is a Turing machine that halts on input } w \}.
\]

\begin{theorem}[Undecidability of the Halting Problem]
	The set $\mathrm{HALT}$ is undecidable; that is, no Turing machine can decide $\mathrm{HALT}$.
\end{theorem}
We can obtain \Cref{und_tep} by a direct reduction of the Halting Problem.
\begin{proof}[Proof of \Cref{und_tep}]
Given $(M,w)$, effectively construct a deterministic one-tape Turing machine $N_{M,w}$ that starts on the blank tape and does the following.

Here, the work and clock tracks are understood via the standard multi-track encoding for one-tape Turing machines.
\begin{enumerate}[label=(\alph*)]
\item It writes the finite word $w$ on a work track.
\item It simulates $M(w)$ step by step.
\item After each completed simulated step, it appends one fresh marker to a clock track.
\item If the simulated machine $M(w)$ halts, then $N_{M,w}$ enters a halting state, which is absorbing under the convention of \Cref{def_ct}.
\end{enumerate}

If $M(w)$ halts, then $N_{M,w}$ eventually enters a halting configuration, and hence its trajectory is eventually fixed.

Suppose $M(w)$ does not halt.
Then $N_{M,w}$ completes infinitely many simulated steps and appends infinitely many simulated steps and appends infinitely many clock markers.
Let $c(C_t)$ denote the number of the fresh markers in the configuration of $N_{M,w}$ at time $t$.
Then $c(C_t)$ is unbounded.

If the trajectory were eventually periodic, then there would exist $k<l$ with $C_k=C_l$.
Determinism would imply
\[C_{k+n}=C_{l+n}\qquad \textrm{for all}\ n\ge0.\]
Hence the tail would be periodic with period $l-k$, and the set of values of $c(C_t)$ along the tail would be finite and therefore bounded.
This contradicts the unbounded growth of the number of the fresh markers.
Thus, if $M(w)$ does not halt, the trajectory of $N_{M,w}$ is not eventually periodic.

We have effectively constructed $N_{M,w}$ such that
\[M(w)\ \text{halts}\quad\longleftrightarrow\quad N_{M,w}\ \text{has an eventually periodic trajectory.}\]
An algorithm for Turing trajectory eventual periodicity would therefore decide the Halting Problem, a contradiction.
\end{proof}

\section{Direct polynomial encoding of Turing machines}

\subsection{Alphabet encoding}
Let $N$ be a deterministic one-tape Turing machine.
Let $A_N$ be the finite set of cell symbols used to encode configurations:
\[A_N=\Gamma\cup(Q\times\Gamma).\]
The symbol $\sqcup\in\Gamma$ is the blank cell with no head.
Enumerate
\[A_N\setminus\{\sqcup\}=\{a_1,\cdots,a_m\}.\]
Define the one-hot encoding
\[v(\sqcup)=0\in\Z^m,\qquad v(a_j)=e_j\in\Z^m\quad(1\le j\le m).\]
A Turing configuration $C$ is encoded as $E(C)\in\Z^\N$ by
\[E(C)_{\iota_m(i,j)}=v(\alpha_i)_j,\]
where $\alpha_i\in A_N$ is the encoded cell symbol at tape position $i$.
Since the tape has finite nonblank support and exactly one head, $E(C)$ is finitely supported.
\begin{lemma}[Injectivity]\label{lem_inj}
The encoding $E$ is injective on valid Turing configurations.
\end{lemma}
\begin{proof}
The block at cell $i$ determines the symbol $a_i\in A_N$, because $v$ is injective on $A_N$.
Therefore all tape symbols, the unique head position, and the state are recovered from $E(C)$.
\end{proof}

\subsection{Radius-one local rule}
The Turing transition induces a radius-one local rule
\[\rho_N:A_N^3\to A_N.\]
For valid configurations, $\rho_N$ gives exactly the next cell symbol after one step of $N$.
On invalid triples, define $\rho_N$ arbitrarily; the reduction starts only from valid configurations.

The radius-one dependence is standard: the next symbol at a cell depends only on whether the head is currently at that cell or at one of its two neighbors.

\subsection{Indicator polynomials}
For $a\in A_N$, define a polynomial $\chi_a\in\Z[Y_1,\cdots,Y_m]$ by
\[\chi_{\sqcup}(Y_1,\cdots,Y_m)=\prod^m_{s=1}(1-Y_s),\]
and, for $a_j\neq\sqcup,$
\[\chi_{a_j}(Y_1,\cdots,Y_m)=Y_j\prod_{s\neq j}(1-Y_s).\]
On the set $\{0,e_1,\cdots,e_m\}$, these polynomials satisfy
\[\chi_a(v(b))=
\begin{cases}
1, & a=b, \\
0, & a\ne b.
\end{cases}\]

Let $Y^-,Y^0,Y^+$ denote three blocks of $m$ variables.
For $1\le j\le m$, define
\[P_{N,j}(Y^-,Y^0,Y^+)=\sum_{(a_-,a_0,a_+)\in A^3_N}v(\rho_N(a_-,a_0,a_+))_j\chi_{a-}(Y^-)\chi_{a_0}(Y^0)\chi_{a_+}(Y^+).\]
The polynomials have integer coefficients and are effectively computable from the finite transition table of $N$.

The global map $F_N:\Z^\N\to\Z^\N$ is the finitely presented local polynomial map with parameters $r=m,\ R=1,$ and local polynomials $P_{N_1,\cdots,P_m}$.

\subsection{One-hot invariance and simulation correctness}
\begin{lemma}[One-hot local correctness]\label{lem_oisc}
For every triple $(a_-,a_0,a_+)\in A^3_N$,
\[(P_{N,1},\cdots,P_{N,m})(v(a_-),v(a_0),v(a_+))=v(\rho_N(a_-,a_0,a_+)).\]
In particular, if all cell blocks of $x$ lie in $\{0,e_1,\cdots,e_m\}$, then all cell blocks of $F_N(x)$ also lie in $\{0,e_1,\cdots,e_m\}$, provided the block pattern is obtained from applying $\rho_N$ cellwise.
\end{lemma}

\begin{proof}
For each encoded input triple, exactly one product of indicator polynomials is equal to 1, namely the product corresponding to the actual triple $(a_-,a_0,a_+)$.
All other products vanish.
The output is therefore exactly the one-hot encoding of $\rho_N(a_-,a_0,a_+)$.
\end{proof}

\begin{lemma}[Simulation lemma]\label{lem_sim}
For every valid Turing configuration $C$,
\[F_N(E(C))=E(T_N(C)).\]
Consequently, for every $t\ge0$,
\[F^t_N(E(C_0))=E(T^t_N(C_0)).\]
\end{lemma}

\begin{proof}
The coordinate block of $F_N(E(C))$ at cell $i$ is obtained by applying the polynomial local rule to the three encoded blocks at cells $i-1,i,i+1$.
By \Cref{lem_oisc}, this is the one-hot encoding of the local Turing update $\rho_N$.
Since $\rho_N$ agrees with the Turing transition on valid configurations, the resulting global configuration is exactly $E(T_N(C))$.
The iterated identity follows by induction.
\end{proof}

\section{Final reduction}
\begin{proposition}\label{prop_ofrt}
Let $N$ be a deterministic one-tape Turing machine and let $C_0$ be a valid finite-support initial configuration.
Then
\[\mathcal{O}_{F_N}(E(C_0))\text{ is finite}\longleftrightarrow\{T^t_N(C_0):t\ge0\}\text{ is finite.}\]
Equivalently, the polynomial orbit is finite if and only if the Turing trajectory is eventually periodic.
\end{proposition}

\begin{proof}
By \Cref{lem_sim}, every point in the polynomial orbit is the encoding of the corresponding Turing configuration.
By \Cref{lem_inj}, equality of encoded points is equivalent to equality of Turing configurations.
Thus
\[F^k_N(E(C_0))=F^l_N(E(C_0))\quad\longleftrightarrow\quad T^k_N(C_0)=T^l_N(C_0).\]
For a deterministic self-map, the forward orbit is finite if and only if some pair of iterates repeats, which is equivalent to eventual periodicity.
\end{proof}

\begin{proof}[Proof of \Cref{und_fo}]
It suffices to establish the singleton case, and assume that an algorithm deciding orbit finiteness for finitely presented local polynomial maps and finitely supported initial points exists.

Given a pair $(M,w)$ of a Turing machine and its input word, let $C_{init}$ be its initial configuration.
Construct the finitely presented local polynomial map $F_M$ and the finitely supported point
\[x_{M,w}=E(C_{init}).\]
This construction is effective.

By \Cref{prop_ofrt},
\[\mathcal{O}_{F_M}(x_{M,w})\text{ is finite}\quad\longleftrightarrow\quad M(w)\text{ has an eventually periodic trajectory}.\]
Thus the assumed orbit-finiteness algorithm would contradicts \Cref{und_tep}.
\end{proof}

\bibliographystyle{amsplain}
\bibliography{references}

\end{document}